\documentclass[11pt, oneside]{article}
\usepackage[margin=1in]{geometry}
\usepackage{graphicx}
\usepackage{cite}

\usepackage{amssymb}
\usepackage{amsmath}
\usepackage{amsthm}
\usepackage{enumitem}
\usepackage{mathtools}
\DeclarePairedDelimiter{\floor}{\lfloor}{\rfloor}

\usepackage{hyperref}

\theoremstyle{plain}
\newtheorem{theorem}{Theorem}[section]
\newtheorem{lemma}{Lemma}[section]
\newtheorem{corollary}{Corollary}[section]
\newtheorem{proposition}{Proposition}[section]
\newtheorem{conjecture}{Conjecture}
\theoremstyle{remark}
\newtheorem{remark}{Remark}
\theoremstyle{definition}
\newtheorem{definition}{Definition}

\title{Young Graphs: 1089 et Al.}
\author{L. H. Kendrick}

\begin{document}
\maketitle
\begin{abstract}
This paper deals with those positive integers $N$ such that, for given integers $g$ and $k$ with $2\le k<g$, the base-$g$ digits of $kN$ appear in reverse order from those of $N$.  Such $N$ are called \textit{$(g, k)$ reverse multiples}.  Young, in 1992, developed a kind of tree reflecting properties of these numbers; Sloane, in 2013, modified these trees into directed graphs and introduced certain combinatorial methods to determine from these graphs the number of reverse multiples for given values of $g$ and $k$ with a given number of digits.  We prove Sloane's isomorphism conjectures for 1089 graphs and complete graphs, namely that the Young graph for $g$ and $k$ is a 1089 graph if and only if $k+1\mid g$ and is a complete Young graph on $m$ nodes if and only if $\floor*{\gcd(g-k, k^2-1)/(k+1)}=m-1$.  We also extend his study of cyclic Young graphs and prove a minor result on isomorphism and the nodes adjacent to the node $[0, 0]$.
\end{abstract}

\section{Introduction}\label{intro}
This paper studies a certain type of positive integer, defined as follows.
\begin{definition}[Reverse Multiple]
Given $g, k\in\mathbb{Z}$, with $2\le k<g$, a \textit{$(g, k)$ reverse multiple} is a positive integer $N$ such that the base-$g$ representations of $N$ and $kN$ are reverses of each other.  That is, if $(c_{n-1}, c_{n-2}, \ldots, c_1, c_0)_g$ denotes $\sum_{i=0}^{n-1}c_ig^i$, then $N=(a_{n-1}, a_{n-2}, \ldots, a_1, a_0)_g$ and $kN=(a_0, a_1, \ldots, a_{n-2}, a_{n-1})_g$, for some integers $a_i$ with $0\le a_i\le g-1$ and $a_{n-1}$ nonzero.
\end{definition}

Ironically, the most prominent mention of these numbers may well have curtailed their prominence: G. H. Hardy, in \textit{A Mathematician's Apology}, refers to the fact that $1089$ and $2178$ are the only four-digit $(10, k)$ reverse multiples as one ``very suitable for puzzle columns and likely to amuse amateurs,'' but ``not serious'' and ``not capable of any significant generalization.''  While we abstain from judgment on the first two counts, it is clear, after the intervening decades, that Hardy misjudged regarding the third.\footnote{The only refuge for his position being in the nebulous qualifier ``significant.''}  A number of works generalize the problem, a list of most of which may be found in Sloane's \cite{sloane} 2013 paper on the topic; of the four works \cite{hoey, holtS, holtL, holtAsym} on the topic cited here not found by Sloane, three were published after his paper, by the same author.  Hardy's comment, however, while perhaps self-fulfilling by way of his stature, may possess some kernel of truth: Sloane's list of references, while appreciable, is rather short for a problem given such exposure as a mention in Hardy's book, albeit a disparaging one, and the problem does, in fact, appear with some frequency in texts on recreational mathematics or mathematical curiosities (e.g., by Gardner \cite{gardner} or Wells \cite{PengNumb}).

We take an approach to the problem, using certain labeled graphs, begun by Young \cite{youngRM, youngTrees} in 1992 and developed further by Sloane in 2013.  Most serious work takes an approach more algebraic than Young's, and studies the problem for a small, preset number of digits for a general base $n$.  This is largely due to the influence of Sutcliffe \cite{sutcliffe}, who, in 1966, initiating the largest interlinked body of work on the problem, determined all bases $n$ in which there exists a two-digit reverse multiple; Sutcliffe cites no works other than a general number theory reference.  In sequence, Kaczynski \cite{unabomber},\footnote{``Better known for other work.'' \cite{pudwell}} in 1968, proved Sutcliffe's conjecture on three-digit reverse multiples, and Pudwell \cite{pudwell}, in 2007, extended Kaczynski's methods for reverse multiples of two or three digits to those of four and five, the behavior of the latter proving to be rather more complicated than that of the former.  Pudwell also proved that $(b, b-1, g-b-1, g-b)_g$ is a $(g, k)$ reverse multiple when $g=b(k+1)$, as did Klosinski and Smolarski \cite{klosinski} in 1969, independently of the above material (they cite none of it).  Outside of passing mentions in books on recreational mathematics, Sloane found two other treatments of the problem: a 2013 paper by Webster and Williams \cite{WW}, which deals with the case in which $g=10$, and a 1975 paper by Grimm and Ballew \cite{GB}.  The latter, after listing all $(10, k)$ reverse multiples less than $10^9$, presents an algorithm that generates three-digit reverse multiples through a process rather distinct from those of other works, in that it begins with the numbers $r_i$ (defined in Subsection \ref{subsec:eqtns}) and shows how to construct the base $g$, multiplier $k$, and digits $a_i$ of a reverse multiple, while other sources fix $g$, and sometimes $k$, and then construct the digits $a_i$ and numbers $r_i$; Grimm and Ballew also differ from other sources by allowing reverse multiples to have leading zeroes.  Kaczynski notes that Prasert Na Nagara of Kasetsart University in Thailand arrived at results similar to his independently, but no reference is given, and efforts to contact both regarding this work have failed.

More recently, Holt \cite{holtS, holtL, holtAsym} has taken a variant approach to the problem, considering reverse multiples, which he calls ``palintiples,''\footnote{He credits the term to Hoey \cite{hoey}.  Strictly speaking, ``palintiple'' refers to a number that is a multiple of its reverse, rather than a divisor, as with reverse multiples.} with any number of digits. His papers begin the project of classifying palintiples based on ``the structure of their carries,'' a partitioning rationale similar to that of the equivalence classes developed in this paper, on which similarities we comment, and of characterizing specific classes of palintiples. He also extends Pudwell's study of four- and five-digit reverse multiples and middle-digit truncation, and begins the study of ``palinomials'' and palintiple derivation: respectively, polynomials constructed from the digits of a given palintiple and the process of obtaining a new palintiple from the carries of a given palintiple. The latter, as he demonstrates \cite{holtAsym}, may be fruitfully used to study seemingly rather intractable equivalence classes of reverse multiples. Holt's approach is more in the algebraic style of Sutcliffe, Kaczynski, and Pudwell, although he notes the utility of Young's approach, as well as a number of intriguing phenomena among its elements (see Section \ref{sec:futres}).

Young, with two papers in 1992, began studying the problem via a kind of infinite tree, in which edges are labeled with pairs of digits $(a_{n-1-i}, a_i)$ and nodes with pairs of numbers $[r_{n-2-i}, r_i]$, defined in Section \ref{sec:background}, and found a correspondence relating reverse multiples and certain kinds of nodes in the tree.  Sloane reformulated these trees into labeled directed graphs, which he calls \textit{Young graphs}, allowing for the application of the transfer matrix method to find the number of $(g, k)$ reverse multiples with a given number of digits (Section 3 of his paper).  He also discusses data gained from constructing all Young graphs for $g\le 40$.  While Young's approach can be used to exhaustively list all $(g, k)$ reverse multiples for given $g$ and $k$, Sloane's approach more obviously raises the questions regarding isomorphism that are dealt with in this paper, and so will be used here.

Young graphs, not reverse multiples, are the predominant focus of this paper, and it is partly our object to demonstrate the utility of the graphical approach to this problem.  We discuss several types of graphs mentioned by Sloane, proving two of his conjectures on isomorphism, and also include minor results determining two points' images under isomorphism and describing the nodes around $[0, 0]$.  Section \ref{sec:background} restates the pertinent work of Sloane and Young, Section \ref{sec:1089graphs} proves (Theorem \ref{thm:1089graph}) Sloane's conjecture on 1089 graphs, Section \ref{sec:completegraphs} proves (Theorem \ref{thm:completegraph}) Sloane's conjecture on complete graphs, and Section \ref{sec:cyclicgraphs} proves minor results (Theorems \ref{thm:z3} and \ref{thm:z5}) on cyclic graphs on three and five nodes and contains a conjecture related to cyclic graphs.  Section \ref{sec:completegraphs} also contains a result on isomorphism (Theorem \ref{thm:0iso}), a result on the predecessors of the node $[0, 0]$ (Corollary \ref{cor:predsof0}), and a result on two-digit reverse multiples (Proposition \ref{prop:2digit}), which is related to the work of Sutcliffe.  In Section \ref{sec:futres}, we discuss possible further research on the topic, including several isomorphism conjectures (Subsection \ref{subsec:conj}) made from examination of Young graphs for low values of $g$, although of a less complete nature than those of Sloane.

\section{Background on Young Graphs}\label{sec:background}
Young \cite{youngRM} developed a method for exhaustively listing all $(g, k)$ reverse multiples for given values of $g$ and $k$ by generating a labeled infinite tree; Sloane modified this approach so that the tree became a kind of directed graph.  We present Young's method under Sloane's modification.

\subsection{Equations}\label{subsec:eqtns}
Let $N=(a_{n-1}, a_{n-2}, \ldots, a_1, a_0)_g$ be a $(g, k)$ reverse multiple, with $0\le a_i\le g-1$ and $a_{n-1}\neq 0$.  Then $kN=\text{Reverse}_g(N)$, so there are integers $r_0, r_1,\ldots,r_{n-2}$ with $0 \le r_i <g$ such that
\begin{equation*}
\begin{array}{rcl}
ka_0 & = & a_{n-1}+r_0g \\
ka_1+r_0 & = & a_{n-2}+r_1g \\
ka_2+r_1 & = & a_{n-3}+r_2g \\
 & \vdots & \\
ka_{n-2}+r_{n-3} & = & a_1+r_{n-2}g \\
ka_{n-1}+r_{n-2} & = & a_0.
\end{array}
\end{equation*}
These equations are those that arise from the base-$g$ columns multiplication representing the equation $kN=\text{Reverse}_g(N)$, written

\begin{center}
\begin{tabular}{cccccccc}
 & {\scriptsize$(r_{n-2})$} & {\scriptsize$(r_{n-3})$} & {\scriptsize$(r_{n-4})$} & \ldots & {\scriptsize$(r_1)$} & {\scriptsize$(r_0)$} &  \\
 & $a_{n-1}$ & $a_{n-2}$ & $a_{n-3}$ & \ldots & $a_2$ & $a_1$ & $a_0$ \\
$\times$ &     &                    &                    &            &             &              & $k$ \\
\hline
 & $a_0$ & $a_1$ & $a_2$ & \ldots & $a_{n-3}$ & $a_{n-2}$ & $a_{n-1}$.
\end{tabular}
\end{center}
While the $r_i$ are defined between $0$ and $g-1$ inclusive, Young \cite{youngRM} proved that $r_i \le k-1$ for all $i$ and that $r_0>0$.  For convenience we may define $r_{-1}=r_{n-1}=0$ so that for all $i$, $0\le i \le n-1$, we have
\begin{equation}\label{eq:rmdig}
ka_i+r_{i-1}=a_{n-1-i}+r_ig.
\end{equation}
\begin{remark}\label{rem:notation}
We use the notation of Young and Sloane.  Holt uses a similar setup, focusing on the numbers $r_i$, but his notation is closer to that of Sutcliffe, Pudwell, and Kaczynski: $g$, $k$, $n$, $a_i$, and $r_i$ are denoted in his papers by $b$, $n$, $k+1$, $d_{k-i}$, and $c_{i+1}$, respectively.
\end{remark}

\subsection{Generating the Graph $H(g, k)$}\label{subsec:hgraph}
$H(g, k)$ is a kind of directed graph, between a subset of the paths of which and $(g, k)$ reverse multiples Young \cite{youngRM} found a correspondence, which we state in Subsection \ref{subsec:youngsthm}.  The nodes in $H(g, k)$ will be labeled with pairs $[r_{n-2-i}, r_i]$ of the $r_i$s and the edges with pairs $(a_{n-1-i}, a_i)$ of the $a_i$s.  To construct this graph, we take the above equations (\ref{eq:rmdig}) in pairs

\begin{equation}\label{eq:eqnpair}
\begin{array}{rcl}
ka_i+r_{i-1} & = &a_{n-1-i}+r_ig \\
ka_{n-1-i}+r_{n-2-i} & = & a_i+r_{n-1-i}g
\end{array}
\end{equation}
and solve recursively for the pairs $[r_{n-2-i}, r_i]$: given a pair $[r_{n-1-i}, r_{i-1}]$, we can check through all values of $a_i$ and $a_{n-1-i}$ (which must be between $0$ and $g-1$ inclusive) for the pairs $(a_{n-1-i}, a_i)$ that give integer values for $r_{n-2-i}$ and $r_i$ between $0$ and $k-1$ inclusive.  If a pair $(a_{n-1-i}, a_i)$ gives a satisfactory pair $[r_{n-2-i}, r_i]$ then a directed edge, labeled $(a_{n-1-i}, a_i)$, is drawn from the node labeled $[r_{n-1-i}, r_{i-1}]$ to the node labeled $[r_{n-2-i}, r_i]$ (with one exception discussed below).  As there are a finite number of potential nodes, this process must terminate.

This process begins at $[r_{-1}, r_{n-1}]$, called the \textit{starting node} and indicated as $[[0, 0]]$ with double brackets to distinguish it from the node labeled $[0, 0]$.  No edges are to end at the starting node by definition, and no edges leading from the starting node can have labels with $0$s, because this would imply $a_{n-1}=0$ or $a_0=0$, which the definition of $a_{n-1}$ and $a_0$ forbids (although not for Grimm and Ballew).  The node $[0, 0]$ cannot function as the starting node for this reason, as edges from $[0, 0]$ are allowed $0$s, and also for the sake of convenience in Young's correspondence and Sloane's enumerations (Section 3 in his paper; they are not discussed here); this is why we must have a starting node.

$H(g, k)$ is the graph thus generated.

\subsection{Pivot Nodes and Young Graphs}\label{subsec:youngsthm}
The following definition, due to Sloane, is made to condense the statements of Young's \cite{youngRM} main theorems, Theorems 1 and 2 in her first paper.
\begin{definition}[Pivot Nodes]\label{def:pivnodes}
An \textit{even pivot node} is a node of the form $[a, a]$; an \textit{odd pivot node} is a node of the form $[r, s]$ that is a direct predecessor of the node $[s, r]$ (including even pivot nodes with self-loops).  The starting node is not considered a pivot node.
\end{definition}

Note that the edge label between $[r, s]$ and $[s, r]$ must have the form $(a, a)$ because equation (\ref{eq:eqnpair}) implies that the edge label from $[r, s]$ to $[s, r]$ is $((rg-s)/(k-1), (rg-s)/(k-1))$.

Theorems 1 and 2 in Young's \cite{youngRM} first paper link these nodes of $H(g, k)$ to $(g, k)$ reverse multiples, and allow us to read off all $(g, k)$ reverse multiples from $H(g, k)$.  The proofs of those theorems are omitted; the results are stated as they are rephrased by Sloane for directed graphs.

\begin{theorem}[Young's Theorem]\label{thm:youngsthm}
There is a one-to-one correspondence between the $(g, k)$ reverse multiples with an even number of digits and the paths in $H(g, k)$ from the starting node to even pivot nodes and a one-to-one correspondence between the $(g, k)$ reverse multiples with an odd number of digits and the paths in $H(g, k)$ from the starting node to odd pivot nodes.  The path leading to an even pivot node that consists of the edges labeled $(a_{n-1}, a_0), (a_{n-2}, a_1),\ldots,(a_{n/2}, a_{n/2-1})$, in that order, corresponds to the reverse multiple $(a_{n-1}, a_{n-2}, \ldots, a_1, a_0)_g$.  The path leading to the odd pivot node $[r, s]$, which directly precedes $[s, r]$ by the edge labeled $(a_{(n-1)/2}, a_{(n-1)/2})$, that consists of the edges labeled $(a_{n-1}, a_0), (a_{n-2}, a_1),\ldots,(a_{(n+1)/2}, a_{(n-3)/2})$ corresponds to the reverse multiple $(a_{n-1}, a_{n-2}, \ldots, a_1, a_0)_g$.
\end{theorem}

Nodes of which no pivot nodes are successors are therefore unimportant, if we are interested in Young's correspondence between paths and reverse multiples.  Thus the Young graph is defined as follows.

\begin{definition}[Young Graph]\label{def:ygraph}
The \textit{Young graph} for $g$ and $k$, denoted $Y(g, k)$, is the labeled directed graph obtained by removing from $H(g, k)$ all nodes that are not pivot nodes and that do not precede any pivot nodes, along with all edges starting or ending at these nodes.
\end{definition}
\begin{remark}\label{rem:truncate}
Sutcliffe, Kaczynski, Pudwell, and, most generally, Holt discuss \textit{middle digit truncation} of reverse multiples; that is, the question of when removing the middle digit of a $(g, k)$ reverse multiple with an odd number of digits yields another $(g, k)$ reverse multiple.  These discussions arose from Sutcliffe's Theorem 3, which found a correspondence between two-digit reverse multiples and three-digit reverse multiples involving such truncation: namely, that if $(a, b)_g$ is a $(g, k)$ reverse multiple, with $0\le a, b\le g-1$, then $(a, a+b, b)_g$ is also a $(g, k)$ reverse multiple.  Theorem 10 of Holt's \cite{holtL} first paper, which characterizes the occurrences of truncations that yield new reverse multiples, states that a $(g, k)$ reverse multiple $(a_{n-1}, a_{n-2}, \ldots, a_1, a_0)_g$, for odd $n$, may be truncated by middle digit removal to obtain a new $(g, k)$ reverse multiple if and only if $r_{(n-3)/2}=r_{(n-1)/2}$; while Holt's proof is algebraic, this result may also be seen as a direct consequence of Young's theorem and the fact that the edge labels $(a_{n-1-i}, a_i)$ in a Young graph are distinct (proven by Sloane), because such reverse multiples correspond to paths to nodes that are both odd and even pivot nodes.
\end{remark}

\subsection{Examples}\label{subsec:ex}

We include several examples to clarify the above material. Consider $Y(10, 9)$, shown in Figure \ref{fig:y109}; to construct this graph, we must first construct $H(10, 9)$ and then remove all nodes that do not precede a pivot node, as per Definition \ref{def:ygraph}. To construct the graph $H(10, 9)$, one begins at the starting node, taking equation (\ref{eq:eqnpair}) with $g=10$, $k=9$, and $i=0$ to obtain
\begin{equation*}
\begin{array}{rcl}
9a_0+0 & = & a_{n-1}+10r_0 \\
9a_{n-1}+r_{n-2} & = & a_0+10\cdot 0.
\end{array}
\end{equation*}
Note that here the value $0$ has been substituted for the terms $r_{i-1}$ and $r_{n-1-i}$ that appear in (\ref{eq:eqnpair}), since $r_{-1}=r_{n-1}=0$ by definition (Subsection \ref{subsec:eqtns}).  Checking through all pairs $(a_{n-1}, a_0)$ with $0<a_0, a_{n-1}\le g-1=9$, one finds that the only pair that gives integer values for $r_0$ and $r_{n-2}$ between $0$ and $k-1$ inclusive is the pair $(1, 9)$, and the corresponding pair $[r_{n-2}, r_0]$ is $[0, 8]$. Therefore we draw a directed edge from $[[0, 0]]=[r_{n-1}, r_{-1}]$ to $[0, 8]=[r_{n-2}, r_0]$ labeled with the pair $(1, 9)=(a_{n-1}, a_0)$. We now repeat the above process with $i=1$, substituting the values for $r_0$ and $r_{n-2}$ into equation (\ref{eq:eqnpair}) to obtain
\begin{equation}\label{eq:ex1}
\begin{array}{rcl}
9a_1+8 & = & a_{n-2}+10r_1 \\
9a_{n-2}+r_{n-3} & = & a_1+10\cdot 0.
\end{array}
\end{equation}
Checking through all possible $(a_{n-2}, a_1)$, one finds the only solution to (\ref{eq:ex1}) to be $(a_{n-2}, a_1)=(0, 8)$ and $[r_{n-3}, r_1]=[8, 8]$. The process is then repeated with $i=2, 3\ldots$ until no new nodes are found, and the resulting graph is shown in Figure \ref{fig:y109}. Note that $[0, 0]$ and $[8, 8]$ are both even and odd pivot nodes, and that all nodes in the graph precede both $[0, 0]$ and $[8, 8]$. Therefore, for $g=10$ and $k=9$, all nodes in the $H$ graph precede a pivot node, so $H(10, 9)=Y(10, 9)$. For general $g$ and $k$, one removes all nodes in $H(g, k)$ that do note precede pivot nodes that are not themselves pivot nodes to obtain $Y(g, k)$, as described in Definition \ref{def:ygraph}.
\begin{figure}[ht]
\centering{
\includegraphics[scale=0.8]{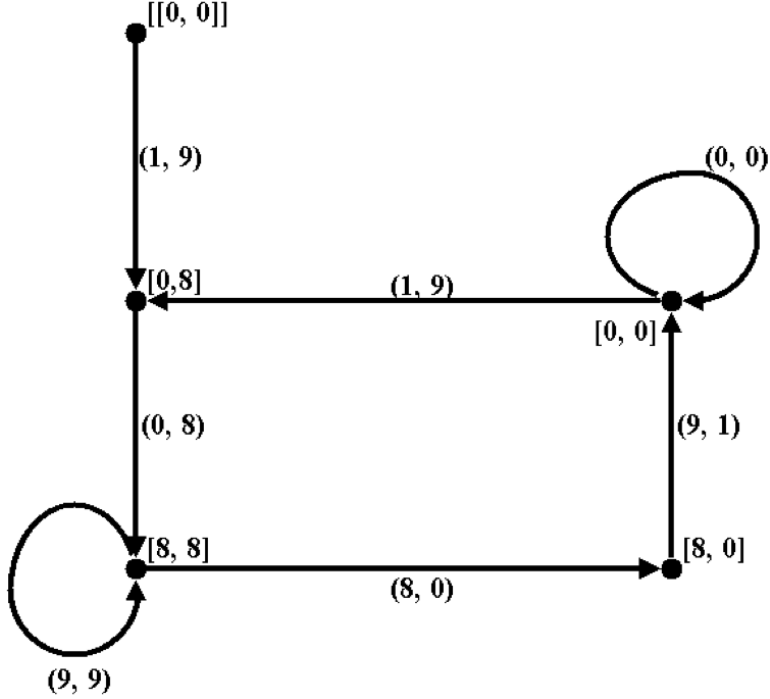}
\caption{$Y(10, 9)$.  The nodes $[0, 0]$ and $[8, 8]$ are both even and odd pivot nodes.}\label{fig:y109}}
\end{figure}

\begin{figure}[ht]
\centering
\includegraphics[scale=0.5]{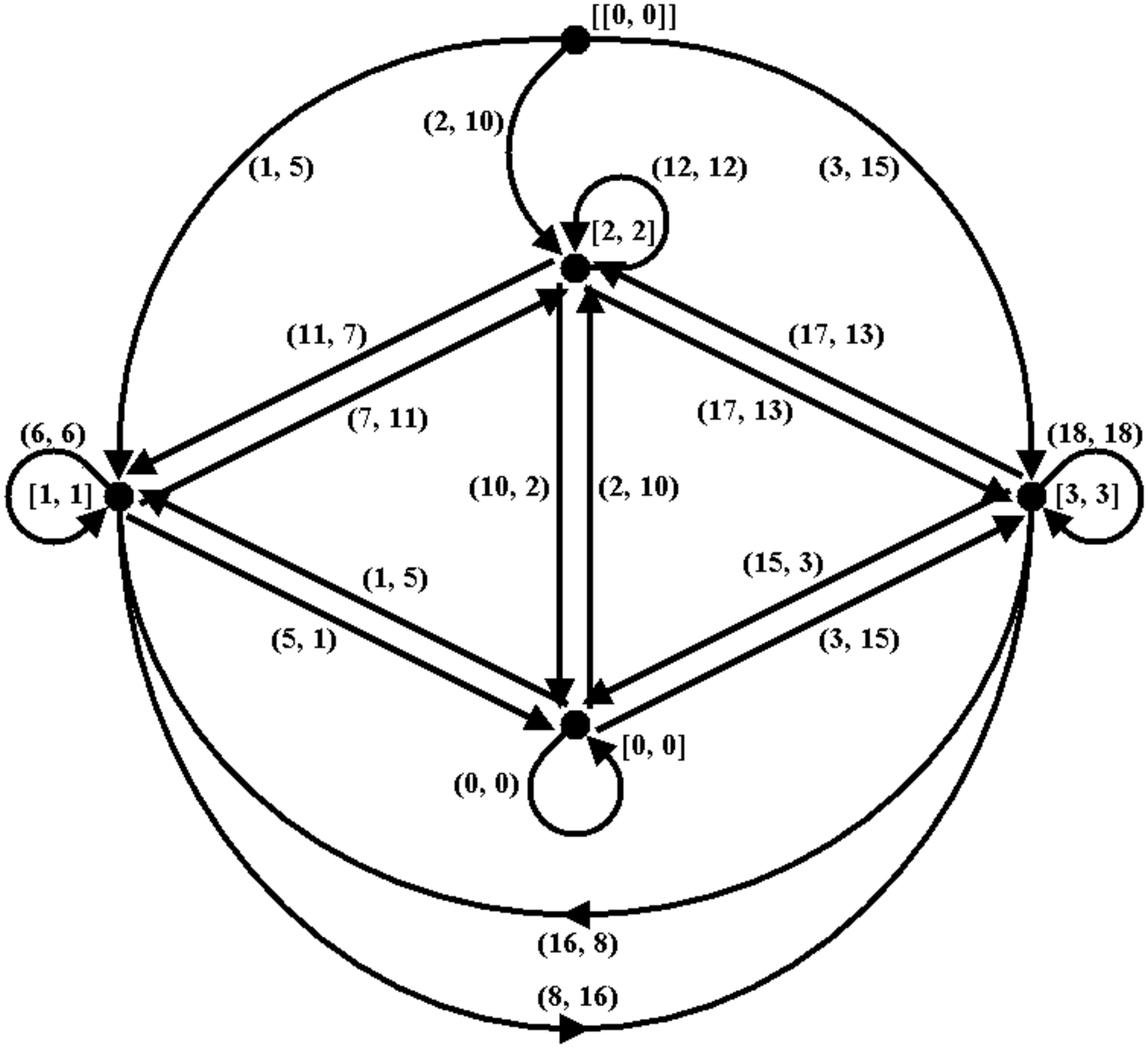}
\caption{$Y(19, 4)$.  All nodes except the starting node $[[0, 0]]$ are both even and odd pivot nodes.}\label{fig:y194}
\end{figure}

For an illustration of Young's theorem (Theorem \ref{thm:youngsthm}) for even pivot nodes, consider $Y(19, 4)$, shown in Figure \ref{fig:y194}.  The node $[2, 2]$ is an even pivot node (by Definition \ref{def:pivnodes}), and there is a path to $[2, 2]$ from the starting node that consists of the edge labeled $(2, 10)$. By Young's theorem, the number $(2, 10)_{19}$ should be a $(19, 4)$ reverse multiple, and, in fact, computation verifies that $4\cdot(2, 10)_{19}=(10, 2)_{19}$.  There is also a path from the starting node to $[2, 2]$ that consists of the edges labeled $(1, 5)$, $(8, 16)$, and $(17, 13)$, in that order; here, Young's theorem tells us that $(1, 8, 17, 13, 16, 5)_{19}$ should be a $(19, 4)$ reverse multiple, and computation will verify that $4\cdot(1, 8, 17, 13, 16, 5)_{19}=(5, 16, 13, 17, 8, 1)_{19}$.  Beyond the fact that reverse multiples may be read off the Young graph in such a way, Young's theorem also states that such reverse multiples constitute \textit{all} reverse multiples.

For an illustration of Young's theorem for odd pivot nodes, consider $Y(14, 3)$, shown in Figure \ref{fig:y143}. The node $[1, 0]$ is an odd pivot node, as it directly precedes $[0, 1]$. There is a path from the starting node to $[1, 0]$ that consists of the edges labeled $(1, 5)$ and $(10, 3)$; furthermore, $[0, 1]$ directly precedes $[1, 0]$ by the edge $(7, 7)$. Therefore, by Young's theorem, the number $(1, 10, 7, 3, 5)_{14}$ should be a $(14, 3)$ reverse multiple, and computation confirms that this is true.
\begin{figure}[ht]
\centering
\includegraphics[scale=0.5]{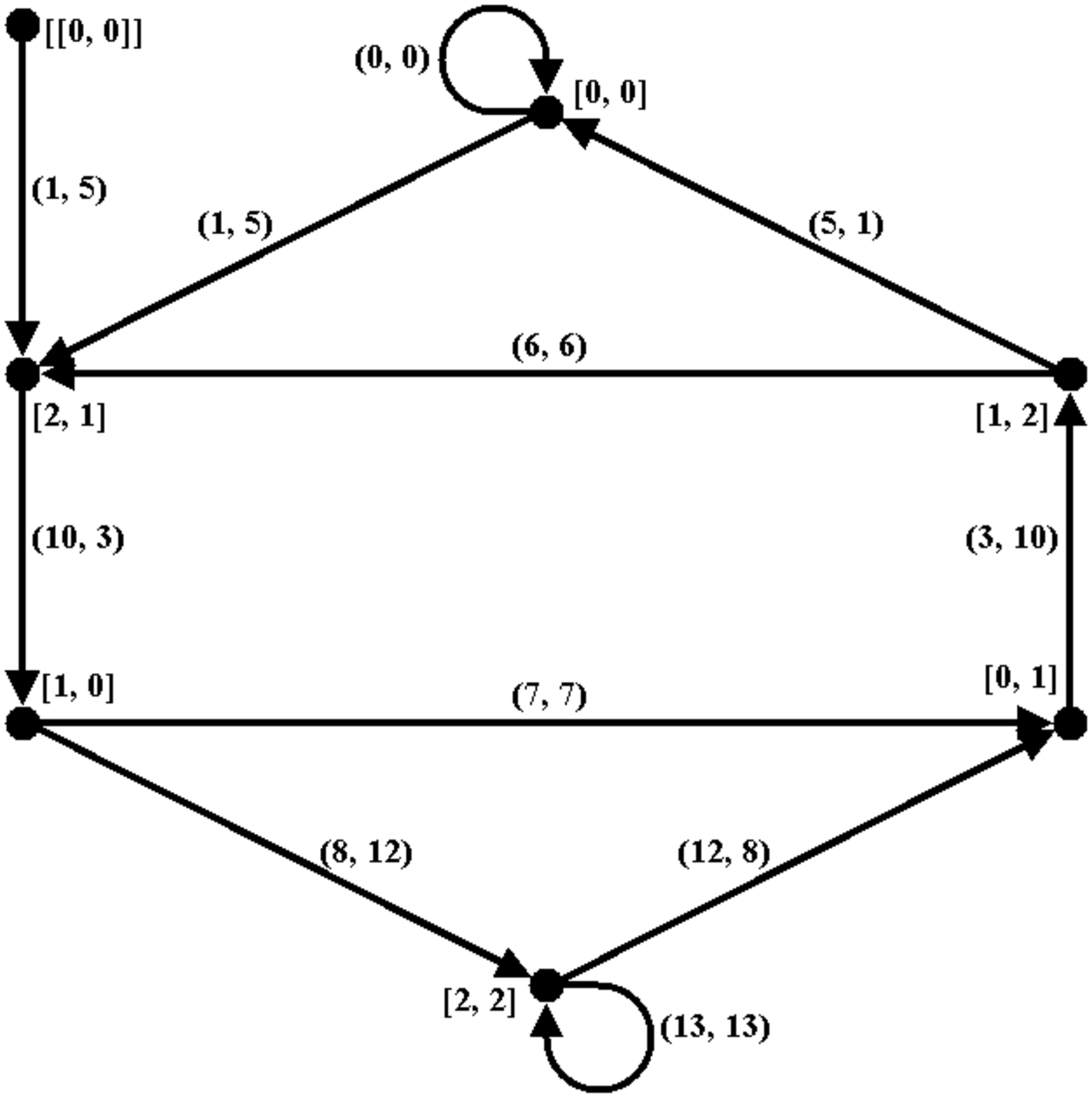}
\caption{$Y(14, 3)$. The nodes $[1, 0]$, $[2, 1]$, $[2, 2]$ and $[0, 0]$ are odd pivot nodes, and $[2, 2]$ and $[0, 0]$ are also even pivot nodes.}\label{fig:y143}
\end{figure}

Thus, more informally, the essence of Young's theorem is that, given a path from the starting node to an even pivot node, one can construct a reverse multiple with an even number of digits from left to right by first reading off the left components of the edge labels in the order they appear as one goes from the starting node to the pivot node, and then reading off the right components of the edge labels in the order they appear as one goes from the pivot node back to the starting node; and that, given a path from the starting node to an odd pivot node, the same procedure gives a reverse multiple with an odd number of digits, after one inserts in the middle of the constructed number the single value that appears in the label for the edge from the odd pivot node to its reverse (e.g., the number $7$ from the edge between $[1, 0]$ and $[0, 1]$ in $Y(14, 3)$, shown in Figure \ref{fig:y143}).

In this way, we may read off all reverse multiples for given $g$ and $k$. For $g=10$ and $k=9$ (see Figure \ref{fig:y109}), this gives the list $1089$, $10989$, $109989$, $1099989$, $10891089$, $10999989$, $108901089\ldots$, which is entry A001232 in the On-Line Encyclopedia of Integer Sequences (OEIS) \cite{oeis}. The other nontrivial base $10$ case, that is, $g=10$ and $k=4$, gives entry A008918, and any of the graphs discussed or shown in this paper will give a similar sequence; Figure \ref{fig:y194} gives the (base $10$) sequence $24$, $48$, $72$, $480$, $960$, $1440\ldots$, and Figure \ref{fig:y143} gives $67275$, $946335$, $13253175$, $185548935$, $2584503675$, $2597689575\ldots$ (neither are listed in the OEIS \cite{oeis}).

\section{1089 Graphs}\label{sec:1089graphs}

In this section and the two following it we present our results.  This section and Section \ref{sec:completegraphs} both prove conjectures of Sloane (Theorems \ref{thm:1089graph} and \ref{thm:completegraph}), while Section \ref{sec:cyclicgraphs} discusses another type of graph mentioned in his paper.  Section \ref{sec:completegraphs} also contains a result (Theorem \ref{thm:0iso}) on Young graph isomorphism and a result (Corollary \ref{cor:predsof0}) on the nodes adjacent to $[0, 0]$.

The results of this section and of Sections \ref{sec:completegraphs} and \ref{sec:cyclicgraphs} mostly concern isomorphism of Young graphs, which is a somewhat stronger condition than isomorphism in ordinary directed graphs because of the added structure of node labels in Young graphs.  The concept was introduced by Sloane.

\begin{definition}[Isomorphism]\label{def:iso}
Two Young graphs $Y$ and $Y'$ are \textit{isomorphic as Young graphs}, written $Y \simeq Y'$, if there is a function $\phi$ from the nodes of $Y$ to those of $Y'$ such that, if $G$ and $G'$ are respectively the directed graphs\footnote{These are referred to as ``underlying directed graphs'' by Sloane.} obtained by removing the edge and node labels from $Y$ and $Y'$, $\phi$ acts as an isomorphism of $G$ and $G'$, and furthermore such that $x$ is an even pivot node in $Y$ if and only if $\phi(x)$ is an even pivot node in $Y'$, and $x$ is an odd pivot node in $Y$ if and only if $\phi(x)$ is an odd pivot node in $Y'$ (that is, $\phi$ is a pivot node-preserving isomorphism of the underlying directed graphs).
\end{definition}
\begin{remark}\label{rem:isocarry}
Holt classifies reverse multiples based on the ``structure of their carries,'' ``carries'' being the numbers $r_i$.  Since it is through the labels $[r_{n-1-i}, r_{i-1}]$ that Young graph isomorphism is determined, studying classes of reverse multiples with similar carry structure should be similar to studying reverse multiples with isomorphic Young graphs; indeed, this relation is seen in the correspondence between Holt's study of symmetric and shifted-symmetric reverse multiples and the material of Sections \ref{sec:1089graphs} and \ref{sec:completegraphs}, as will become evident throughout this paper.
\end{remark}

For an example of isomorphic graphs, see Figures \ref{fig:y109} and \ref{fig:1089graph}. It is noted in Section 3.2 of Sloane's paper that Young graphs isomorphic to $Y(10, 9)$ (Figure \ref{fig:y109}) are fairly common, for example $Y(10, 4)$, $Y(15, 4)$, and $Y(20, 9)$; the following definition is then made.

\begin{definition}[1089 Graph]\label{def:1089graph}
A Young graph $Y$ such that $Y\simeq Y(10, 9)$ is called a \textit{1089 graph}.
\end{definition}

The number $1089$ is the smallest $(10, 9)$ reverse multiple; thus the name ``1089 graph.''  Sloane's Conjecture 3.1 characterizes the occurrences of 1089 graphs; we prove this conjecture in Theorem \ref{thm:1089graph}.  The work of this section is also related to Holt's \cite{holtS} work on 1089 palintiples.  We require several preliminary results.

\begin{lemma}\label{lem:edgereverse}
If a Young graph contains a node $[r, s]$ directly preceding a node $[t, u]$ by an edge $(a, b)$ then it also contains the node $[u, t]$ directly preceding the node $[s, r]$ by the edge $(b, a)$.
\end{lemma}
We omit the proof; the interested reader is directed to property P3 in Section 2.7 of Sloane or Theorem 2 of Young \cite{youngTrees}.

\begin{corollary}\label{cor:precede0}
Every node in a Young graph is a predecessor of $[0, 0]$.
\end{corollary}
\begin{proof}
This is a direct consequence of Lemma \ref{lem:edgereverse} and the fact that all nodes in a Young graph precede some pivot node and succeed the starting node.
\end{proof}

Two more lemmas are needed to prove the main theorem; we now prove the first.
\begin{lemma}\label{lem:1089suc}
Let $g=b(k+1)$, with $b$ a positive integer, and let $m$ be a nonnegative integer.  Then, if the $(g, k)$ Young graph contains a node of the form $\left[\frac{s(k-1)}{b^m}, \frac{t(k-1)}{b^{m+1}}\right]$, where $s$ and $t$ are integers, the following are true:

\begin{enumerate}[label=(\roman*)]
\item\label{itm:edges}
The edges leading from $\left[\frac{s(k-1)}{b^m}, \frac{t(k-1)}{b^{m+1}}\right]$ have the form $\left(\frac{s(k-1)}{b^{m-1}}+u, \frac{(t-b^2s)(k-1)}{b^{m+1}}+ku\right)$, where $u$ is some integer.
\item\label{itm:dsuc} The edge $\left(\frac{s(k-1)}{b^{m-1}}+u, \frac{(t-b^2s)(k-1)}{b^{m+1}}+ku\right)$ leading from the node $\left[\frac{s(k-1)}{b^m}, \frac{t(k-1)}{b^{m+1}}\right]$ terminates at the node $\left[\frac{t(k-1)}{b^{m+1}}, \frac{(t-b^2s+b^{m+1}u)(k-1)}{b^{m+2}}\right]$.
\item\label{itm:t0} $t\equiv 0 \pmod{b}$.
\end{enumerate}
\label{lem:1089graph1}
\end{lemma}
\begin{proof}
First note that $0 \le s \le b^m$ and $0 \le t  \le b^{m+1}$, because $[s(k-1)/b^m, t(k-1)/b^{m+1}]$ is in the Young graph and $0\le r_i \le k-1$ for all $i$.
Throughout the following proof let $a=a_i$, $A=a_{n-1-i}$, $r=r_i$, and $R=r_{n-2-i}$, where $i$ is the number such that $[s(k-1)/b^m, t(k-1)/b^{m+1}]=[r_{i-1}, r_{n-1-i}]$.  We will establish the equation $rk=\frac{(b^{m+1}a+t)(k-1)}{b^{m+2}}-\frac{s(k-1)}{b^m}$, whence we directly relate $a$ and $A$ and the desired conclusions \ref{itm:edges}, \ref{itm:dsuc}, and \ref{itm:t0} follow easily.

By equation (\ref{eq:eqnpair}), $A=ka+\frac{t(k-1)}{b^{m+1}}-rg$ and $kA+R=a+\frac{s(k-1)}{b^m}g$.
These two equations imply $R+(k^2-1)a+\frac{tk(k-1)}{b^{m+1}}=g\left(\frac{s(k-1)}{b^m}+rk\right)$.
Because $0\le R \le k-1$, we now have
$$(k^2-1)a+\frac{tk(k-1)}{b^{m+1}} \le g\left(\frac{s(k-1)}{b^m}+rk\right) \le (k^2-1)a+\frac{tk(k-1)}{b^{m+1}}+k-1.$$
Since $g=b(k+1)$, multiplying this inequality by $b^{m+1}/(k+1)$ yields
\begin{equation}\label{ineq:1089bds1}
\begin{array}{rcl}
b^{m+2}\left(\frac{s(k-1)}{b^m}+rk\right) & \ge & b^{m+1}(k-1)a+tk-2t+\frac{2t}{k+1} \\
b^{m+2}\left(\frac{s(k-1)}{b^m}+rk\right) & \le & b^{m+1}(k-1)a+tk-2t+\frac{2t}{k+1}+b^{m+1}-\frac{2b^{m+1}}{k+1}.
\end{array}
\end{equation}

The inequalities in (\ref{ineq:1089bds1}) give lower and upper bounds on $b^{m+2}\left(\frac{s(k-1)}{b^m}+rk\right)$.  Note that the difference between these bounds is $b^{m+1}-\frac{2b^{m+1}}{k+1}<b^{m+1}$, and that $b^{m+2}\left(\frac{s(k-1)}{b^m}+rk\right)$ is a multiple of $b^{m+1}$.  This implies that $b^{m+2}\left(\frac{s(k-1)}{b^m}+rk\right)$ is the only multiple of $b^{m+1}$ satisfying the bounds in (\ref{ineq:1089bds1}).  Now, because $k\ge 2$ and $0\le t\le b^{m+1}$, it may be simply shown that $b^{m+1}(k-1)a+tk-t$ satisfies the bounds in (\ref{ineq:1089bds1}), and because $\frac{t(k-1)}{b^{m+1}}$ is an integer (it is a node label) $b^{m+1}(k-1)a+tk-t$ must be a multiple of $b^{m+1}$.  This establishes $b^{m+1}(k-1)a+tk-t=b^{m+2}\left(\frac{s(k-1)}{b^m}+rk\right)$.

It now follows that $b^{m+2}\left(\frac{s(k-1)}{b^m}+rk\right)=b^{m+1}(k-1)a+tk-t,$ which implies
\begin{equation}\label{eq:1089eq1}
rk=\frac{(b^{m+1}a+t)(k-1)}{b^{m+2}}-\frac{s(k-1)}{b^m}.
\end{equation}

Equation (\ref{eq:eqnpair}) implies $kA+rkg=k^2a+\frac{kt(k-1)}{b^{m+1}}$.  Substituting for $rk$ via equation (\ref{eq:1089eq1}) and substituting $b(k+1)$ for $g$, we obtain

\begin{equation}\label{eq:1089eq2}
a=kA+\frac{t(k-1)}{b^{m+1}}-\frac{s(k^2-1)}{b^{m-1}}.
\end{equation}

Let $u=A-\frac{s(k-1)}{b^{m-1}}\in\mathbb{Z}$.  Substituting $u$ into equation (\ref{eq:1089eq2}) yields $a=\frac{(t-b^2s)(k-1)}{b^{m+1}}+ku$, establishing \ref{itm:edges}.  Now equation (\ref{eq:eqnpair}) gives $A+rg = ka+\frac{t(k-1)}{b^{m+1}}$ and $a+\frac{s(k-1)}{b^m}g = kA+R$, and
substituting in for $A$ and $a$ in terms of the parameter $u$ and for $g$ in terms of $b$ and $k$ yields \ref{itm:dsuc}.

Note that the new node $[t(k-1)/b^{m+1}, (t-b^2s+b^{m+1}u)(k-1)/b^{m+2}]$ has the same form as the directly preceding node $[s(k-1)/b^m, t(k-1)/b^{m+1}]$.  Thus all successors of the node $[s(k-1)/b^m, t(k-1)/b^{m+1}]$ have the form $[s'(k-1)/b^{m'}, t'(k-1)/b^{m'+1}]$.
Additionally, we have $t\equiv {t-b^2s+b^{m+1}u} \pmod{b}$, so that $t\equiv s'\equiv t'\pmod{b}$, inductively.
Now, if $[s(k-1)/b^m, t(k-1)/b^{m+1}]$ is in $Y(g, k)$, it must precede $[0, 0]$ by Corollary \ref{cor:precede0}, so there must be some $t'=0$, which implies \ref{itm:t0}.
\end{proof}

One further lemma is required before we may characterize all occurrences of the 1089 graph.
\begin{lemma}\label{lem:1089graph2}
If, in $Y(g, k)$, a node $[r, s]$ directly precedes $[0, 0]$ and is not an odd pivot node then $k+1 \mid g$.
\end{lemma}
\begin{proof}
By Lemma \ref{lem:edgereverse}, $[0, 0]$ must directly precede $[s, r]$; let the edge from $[0, 0]$ to $[s, r]$ be labeled $(a, b)$.  We then have $kb=a+rg$ and $ka+s=b$, by equation (\ref{eq:eqnpair}); adding these equations and rearranging yields $(k-1)(a+b)=rg-s$.

It is also known that there can be no edge from $[r, s]$ to $[s, r]$.  If we attempt to find such an edge by solving equation (\ref{eq:eqnpair}) for the edge $(x, y)$ from $[r, s]$ to $[s, r]$, we find $x=y=(rg-s)/(k-1)=a+b$, and so $a+b$ must be an unsatisfactory edge label.  It is only required of edge labels that they be integers between 0 and $g-1$ inclusive; $a+b$ is a nonnegative integer, so $a+b \ge g$.  Now by Sloane's Theorem 3.4, we have $k+1\mid g$.
\end{proof}

We may now prove the main theorem of this section.
\begin{theorem}[Conjecture 3.1 of Sloane]\label{thm:1089graph}
$Y(g, k)$ is a 1089 graph if and only if $k+1\mid g$.  Also, if $Y(g, k)$ is a 1089 graph, with $g=b(k+1)$, its labels are those shown in Figure \ref{fig:1089graph}.

\begin{figure}[ht]
\centering
\includegraphics[scale=0.7]{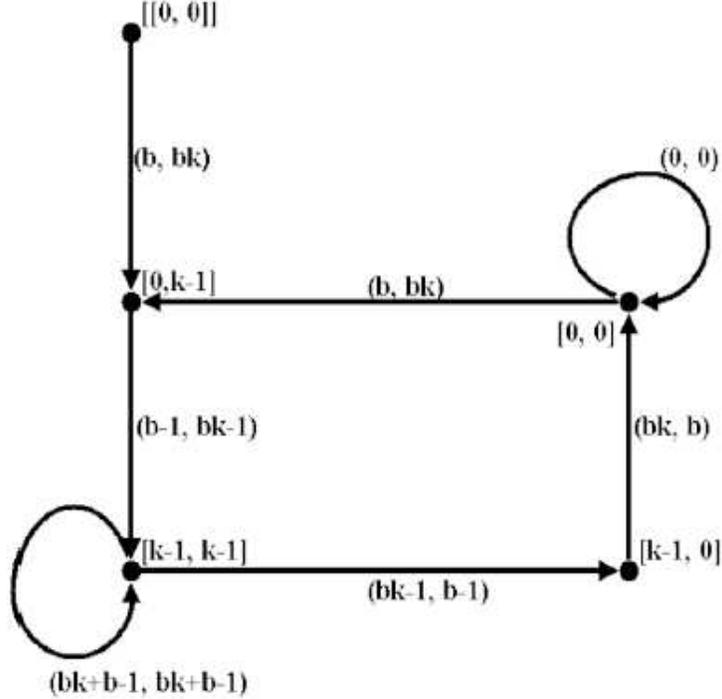}
\caption{$Y(b(k+1), k)$}\label{fig:1089graph}
\end{figure}
\end{theorem}
\begin{proof}
We first show that $k+1| g$ implies $Y(g, k)$ is a 1089 graph, and then the converse. The values of the edge and node labels in Figure \ref{fig:1089graph} are established in the first part of the proof.

Suppose $g=b(k+1)$.  We may now determine $Y(g, k)$ by repeated application of Lemma \ref{lem:1089graph1}.  Consider the starting node: $[[0, 0]]=[[0(k-1)/b^0, 0(k-1)/b^1]]$, so by Lemma \ref{lem:1089graph1}.\ref{itm:dsuc} the starting node can only directly precede nodes of the form $[0(k-1)/b, bu(k-1)/b^2]=[0(k-1)/b^0, u(k-1)/b^1]$.  Note that if $[0(k-1)/b^0, u(k-1)/b^1]$ is a node, we have $u \equiv 0\pmod{b}$ by Lemma \ref{lem:1089graph1}.\ref{itm:t0}.  We also have $0\le u \le b$ because $0\le r_i\le k-1$ for all $i$, so $u=0$ or $u=b$.  But the edge leading to this first node has the form $(u, ku)$, by Lemma \ref{lem:1089graph1}.\ref{itm:edges}, and $a_0\neq 0$ so $u=b$.  Therefore the starting node directly precedes only the node $[0, k-1]$, by the edge $(b, bk)$.

Similar reasoning shows that $[0, k-1]$ only directly precedes $[k-1, k-1]$, by the edge $(b-1, bk-1)$; that $[k-1, k-1]$ only directly precedes itself, by the edge $(bk-b-1, bk-b-1)$, and $[k-1, 0]$, by the edge $(bk-1, b-1)$; that $[k-1, 0]$ only directly precedes $[0, 0]$, by the edge $(bk, b)$; and that $[0, 0]$ only directly precedes itself, by the edge $(0, 0)$, and $[0, k-1]$, by the edge $(b, bk)$.  The graph $Y(g, k)$ thus determined is shown in Figure \ref{fig:1089graph}, and is a 1089 graph.

Now suppose that $Y(g, k)$ is isomorphic to $Y(10, 9)$ (Figure \ref{fig:y109}).  The starting node then has a direct successor $[r, s]$, with $r \neq s$, since the node is not a pivot node, by the isomorphism.  Again by the isomorphism, this node $[r,s]$ has an even pivot node as its only direct successor, which we call $[t, t]$.  This node $[t, t]$ has a self-loop and directly precedes only one other node; by Lemma \ref{lem:edgereverse} this node must be $[s, r]$ -- note that $[s, r]$ must be distinct from $[t,t]$ because $s\neq r$.  The node $[s, r]$ directly precedes only one other node, and by Lemma \ref{lem:edgereverse} this must be $[0, 0]$, since $[r,s]$ directly succeeds the starting node.  Now $[s, r]$ is not an odd pivot node, because all odd pivot nodes are even pivot nodes in 1089 graphs and $s \neq r$, but it directly precedes $[0, 0]$.   Therefore, by Lemma \ref{lem:1089graph2}, we have $k+1\mid g$.  By the first part of this proof, the $Y(g, k)$ determined is shown in Figure \ref{fig:1089graph}.
\end{proof}

\begin{remark}\label{rem:1089graph}
This theorem is related Holt's work (Theorem 6 in his first paper \cite{holtL} and Conjectures 1 and 3 and Corollary 2 in his second \cite{holtS}) on \textit{symmetric} reverse multiples (defined in his first paper \cite{holtL}): reverse multiples such that $r_{i-1}=r_{n-2-i}$, as is the case with reverse multiples arising from 1089 graphs.  In his second paper, Holt conjectures that a $(g, k)$ reverse multiple is symmetric if and only if $k+1\mid g$.  If true, this conjecture would give a further distinction to the 1089 graph, beyond the number theoretical characterization and the corollaries below: it is the graph for which, if $[a, b]\longrightarrow[c, d]$ is in the graph, then $b=c$.  Note also that, in finding the edge and node labels in 1089 graphs, Theorem \ref{thm:1089graph} also establishes part of Theorem 3.2 in Sloane's paper, which in addition finds the generating function for the number of reverse multiples with a given number of digits and characterizes the associated reverse multiples for 1089 graphs; Holt \cite{holtS} makes a related conjecture (Conjecture 3) for the broader problem of symmetric reverse multiples.
\end{remark}

It is a consequence of Theorem \ref{thm:1089graph} that, for $g=b(k+1)$, all $(g, k)$ reverse multiples are contained in the list $(b, b-1, bk-1, bk)_g, (b, b-1, bk+b-1, bk-1, bk)_g, (b, b-1, bk+b-1, bk+b-1, bk-1, bk)_g, (b, b-1, bk+b-1, bk+b-1, bk+b-1, bk-1, bk)_g, (b, b-1, bk-1, b-1, bk, b, b-1, bk-1, bk)_g\ldots$ that may be read off Figure \ref{fig:1089graph} in the manner described in Subsection \ref{subsec:ex}. Furthermore, several minor results follow immediately from Theorem \ref{thm:1089graph}.

\begin{corollary}\label{cor:oddpred0}
In a Young graph that is not a 1089 graph ($k+1 \nmid g$), all direct predecessors of $[0, 0]$ are odd pivot nodes.
\end{corollary}
\begin{proof}
This follows from Lemma \ref{lem:1089graph2} and Theorem \ref{thm:1089graph}.
\end{proof}

\begin{corollary}\label{cor:oddinall}
All Young graphs contain an odd pivot node other than $[0, 0]$.
\end{corollary}
\begin{proof}
Corollary \ref{cor:precede0} implies that the node $[0, 0]$ is in all Young graphs $Y(g, k)$ (also see property P4 in Section 2.7 in Sloane \cite{sloane}).   As it is in the Young graph, the node $[0, 0]$ must have some direct predecessor $[r, s] \neq [0, 0]$.  If $[r, s]$ is not an odd pivot node then the Young graph is a 1089 graph, by Corollary \ref{cor:oddpred0}, in which case it contains an odd pivot node by definition.
\end{proof}

\begin{corollary}[Conjecture 3.5 of Sloane]\label{cor:lastdig1089}
If there is a $(g, k)$ reverse multiple with first and last digits that sum to $g$, $Y(g, k)$ is a 1089 graph.  If $Y(g, k)$ is a 1089 graph then every $(g, k)$ reverse multiple has first and last digits summing to $g$.
\end{corollary}
\begin{proof}
By Sloane's \cite{sloane} Theorem 3.4, if $a_0+a_{n-1}=g$ then we must have $a_{n-1}=\frac{g}{k+1}$, so that $Y(g, k)$ is a 1089 graph by Theorem \ref{thm:1089graph}.  If $Y(g, k)$ is a 1089 graph then, by Theorem \ref{thm:1089graph}, $a_0=\frac{gk}{k+1}$ and $a_{n-1}=\frac{g}{k+1}$, and these numbers sum to $g$.
\end{proof}

Isomorphism is an equivalence relation on Young graphs, and thus defines a corresponding equivalence relation on pairs of integers $(g, k)$: we take $(g, k) \sim (g', k')$ if and only if $Y(g, k) \simeq Y(g', k')$.  Via these corresponding relations, any equivalence class $X$ of Young graphs corresponds to an equivalence class\footnote{Note that two degenerate Young graphs -- graphs for which $H(g, k)$ contains no pivot nodes -- are isomorphic.  The equivalence class of $(g, k)$ pairs corresponding to the equivalence class of degenerate graphs is of great interest, as characterizing this equivalence class is equivalent to characterizing all $g$ and $k$ for which reverse multiples exist.} $C$ of pairs $(g, k)$, and so for any equivalence class $X$ of Young graphs we can also define a relation $R_X(g, k)$ on $g$ and $k$ to be the relation that holds if and only if $(g, k)\in C$; this relation sometimes coincides with a convenient, simply expressed arithmetical condition, as is the case for 1089 graphs and the complete graphs discussed in Section \ref{sec:completegraphs}.  Theorem \ref{thm:1089graph} has a somewhat interesting consequence regarding these relations and equivalence classes.

\begin{corollary}\label{cor:equivclassdiv}
The set of 1089 graphs is the only equivalence class $X$ of Young graphs such that $R_X(g, k)$ holds if and only if $f(k) \mid g$, for some fixed function $f:\mathbb{Z}^+\to\mathbb{Z}$.
\end{corollary}
\begin{proof}
Suppose there is such an equivalence class $X$ of Young graphs.  Then $Y(18|f(17)|, 17)$ is in $X$, but by Theorem \ref{thm:1089graph} is also a 1089 graph, so $X$ is the set of 1089 graphs.
\end{proof}

While not related to 1089 graphs, the following result comes from a line of reasoning similar to that of Lemma \ref{lem:1089graph2}, and so we include it in this section.

\begin{proposition}\label{prop:oddthrueven}
If a node $[r, s]$ directly precedes, by an edge $(a, b)$, a node $[t, t]$ that has a self-loop $(c, c)$, $[r, s]$ is an odd pivot node if and only if $0 \le a+b-c \le g-1$, and if it is an odd pivot node then the edge from $[r, s]$ to $[s, r]$ is labeled $(a+b-c, a+b-c)$.
\end{proposition}
\begin{proof}
Taking equation (\ref{eq:eqnpair}) for the given edge $(a, b)$ gives $kb+s=a+tg$ and $ka+t=b+rg$, yielding $(k-1)(a+b)=rg-s+(g-1)t$.  Taking equation (\ref{eq:eqnpair}) for the self-loop at $[t, t]$ gives $(g-1)t=(k-1)c$, so that $(rg-s)/(k-1)=a+b-c$.  The edge from $[r, s]$ to $[s, r]$ can be found, from equation (\ref{eq:eqnpair}), to be $((rg-s)/(k-1), (rg-s)/(k-1))=(a+b-c, a+b-c)$.  So if $0 \le a+b-c \le g-1$ then there is an edge from $[r, s]$ to $[s, r]$, and it is labeled $(a+b-c, a+b-c)$, and if there is an edge from $[r, s]$ to $[s, r]$ it must be $(a+b-c, a+b-c)$ and so $0 \le a+b-c \le g-1$.
\end{proof}

\section{Complete Graphs}\label{sec:completegraphs}
We now turn to the proof of another of Sloane's conjectures (Conjecture 3.7 in his paper), concerning complete Young graphs.

\begin{definition}[Complete Graph]\label{def:completegraph}
A Young graph $Y$ is called a \textit{complete Young graph on $m$ nodes} when the nodes that are not $[[0,0]]$ form the complete directed graph on $m$ nodes and there is an edge from $[[0,0]]$ to every node except for $[0,0]$. The equivalence class of complete Young graphs on $m$ nodes is denoted $K_m$.
\end{definition}

Figure \ref{fig:y194} shows an example of a complete graph, $Y(19, 4)\in K_4$. Sloane's conjecture characterizes all $(g, k)$ that correspond to graphs in $K_m$; we prove this conjecture in Theorem \ref{thm:completegraph}.  Several preliminary results are needed, some of which were proven by Sloane, and several other interesting digressions from the course of the proof are included.

\begin{lemma}\label{lem:0inall}
All Young graphs contain the node $[0,0]$.
\end{lemma}
\begin{lemma}\label{lem:evencomplete}
A Young graph is a complete graph if and only if every node has the form $[r, r]$.
\end{lemma}
We omit the proofs; the interested reader is directed to property P4 in Section 2.7 (or even Corollary \ref{cor:precede0} here) for Lemma \ref{lem:0inall} and to Theorem 3.6 for Lemma \ref{lem:evencomplete} of Sloane's paper.  We now begin work towards Theorem \ref{thm:completegraph}.

\begin{lemma}\label{lem:1intcomplete}
For given $g$ and $k$, if there is a positive integer $s\le k-1$ such that $s\left(\frac{g-k}{k^2-1}\right)$ is an integer then $Y(g, k)$ is a complete Young graph.
\end{lemma}
\begin{proof}
We show that all nodes in $Y(g, k)$ have the form $[r, r]$, so that the conclusion follows from Lemma \ref{lem:evencomplete}.  

Consider a node $[r, r]$ in $Y(g, k)$.  Suppose this node directly precedes some node $[r_{n-2-i}, r_i]$ by an edge $(a_{n-1-i}, a_i)$.  We can solve equation (\ref{eq:eqnpair}) for this edge to yield
$$a_i=\frac{rg-rk+r_ikg-r_{n-2-i}}{k^2-1}=r\frac{g-k}{k^2-1}+r_i\frac{kg-1}{k^2-1}+\frac{r_i-r_{n-2-i}}{k^2-1}.$$
Therefore, $s\left(\frac{r_i-r_{n-2-i}}{k^2-1}\right)=sa_i-r\left(s\frac{g-k}{k^2-1}\right)-r_i\left(ks\frac{g-k}{k^2-1}+s\right)$ is an integer.  Now
$$s\frac{|r_i-r_{n-2-i}|}{k^2-1}\le(k-1)\frac{|r_i-r_{n-2-i}|}{k^2-1} \le (k-1)\frac{k-1}{k^2-1}<1.$$
Therefore $s\frac{|r_i-r_{n-2-i}|}{k^2-1}=0$, so $r_i=r_{n-2-i}$, because $s>0$.

Repeating this argument proves that all successors of a node $[r, r]$ have the form $[r^\prime, r^\prime]$; since all nodes in a Young graph are successors of $[[0, 0]]$, all nodes in $Y(g, k)$ must have the form $[r, r]$, so that $Y(g, k)$ is complete by Lemma \ref{lem:evencomplete}.
\end{proof}
\begin{remark}\label{rem:shiftsym1}
This is related to Theorem 9 in Holt's first paper \cite{holtL} and Corollary 6 in his second \cite{holtS}.  Note as well that here, as with 1089 graphs, Holt's partition based on the structure of the carries of reverse multiples aligns nicely with partition by isomorphism, as noted in Remark \ref{rem:isocarry}.
\end{remark}

While not the full result characterizing the occurrences of complete graphs, the following result is still interesting, and gives several further results, and so we digress from the proof of Theorem \ref{thm:completegraph} to include it here.

\begin{proposition}\label{prop:2evenpreccomplete}
If there are two even pivot nodes in $Y(g, k)$ that directly precede the same node, $Y(g, k)$ is a complete graph.
\end{proposition}
\begin{proof}
Suppose there is a node $[x, x]$ which directly precedes the nodes $[r, s]$ by the edge labeled $(a, b)$ and a node $[y, y]$ that directly precedes $[r, s]$ by the edge labeled $(c, d)$, with $x \neq y$.  Then, by equation (\ref{eq:eqnpair}), we must have $kb+x=a+sg$, $ka+r=b+gx$, $kd+y=c+sg$, and $kc+r=d+gy$.  These yield $y-x=(c-a)-k(d-b)$ and $g(y-x)=k(c-a)-(d-b),$ which together give $(g-k)(c-a)=(kg-1)(d-b)$.  This rearranges to
$$d-b=\frac{g-k}{k^2-1}((c-a)-k(d-b))=\frac{g-k}{k^2-1}(y-x),$$
so that $c-a=\left(\frac{kg-1}{k^2-1}\right)(y-x)$.  Therefore $\left(\frac{g-k}{k^2-1}\right)|y-x|=|d-b|$ is a positive integer and $|y-x|=\left(\frac{k^2-1}{kg-1}\right)|c-a|\le\left(\frac{k^2-1}{kg-1}\right)g<k$, so, by Lemma \ref{lem:1intcomplete}, $Y(g, k)$ is a complete graph.
\end{proof}
Note that this proposition applies to the starting node, as well as all even pivot nodes, so long as the other node is not $[0, 0]$.

\begin{corollary}\label{cor:evennear0}
If, in a Young graph $Y$, an even pivot node $[s, s]$, with $s \neq 0$, directly precedes or succeeds $[0, 0]$ (the two are equivalent by Lemma \ref{lem:edgereverse}), $Y$ is a complete graph.
\end{corollary}
\begin{proof}
Under these conditions $[s, s]$ and $[0, 0]$ both directly precede the same node (the node $[0, 0]$), and are even pivot nodes, so that $Y$ is complete by Proposition \ref{prop:2evenpreccomplete}.
\end{proof}

\begin{corollary}\label{cor:predsof0}
In Young graphs that are neither complete graphs nor 1089 graphs, all direct predecessors of $[0, 0]$ are odd, and not even, pivot nodes.  The direct successors of $[0, 0]$ are not even pivot nodes.
\end{corollary}
\begin{proof}
This follows from Corollaries \ref{cor:oddpred0} and \ref{cor:evennear0}.
\end{proof}

This line of reasoning culminates in a general result on isomorphism.

\begin{theorem}\label{thm:0iso}
If $Y$ and $Y'$ are isomorphic as Young graphs by an isomorphism $\phi$ then $\phi ([[0, 0]])=[[0, 0]]$ and $\phi ([0, 0])=[0, 0]$.
\end{theorem}
\begin{proof}
Clearly $\phi ([[0, 0]])=[[0, 0]]$, since $[[0, 0]]$ is the only node which is the successor of no other node.  We must now only show that $\phi([0, 0])=[0, 0]$.  To this end, suppose $\phi([0 ,0])\neq [0, 0]$.  
As a preliminary, note that the direct successors of the starting node are also direct successors of the node $[0, 0]$, because $[0, 0]$ is in the graph by Lemma \ref{lem:0inall} and the equations given by (\ref{eq:eqnpair}) governing the node and edge labels are the same for the starting node and $[0, 0]$.\footnote{It is, however, possible for $[0, 0]$ to directly precede a node $x$ that does not directly succeed the starting node because the left label on the edge from $[0, 0]$ to $x$ is $0$.}

Because $Y$ exists and is connected, it must include an edge $[[0, 0]]\longrightarrow[r, s]$, for some node $[r, s]$.  Since $\phi$ is an isomorphism, the edge $[[0, 0]]\longrightarrow\phi([r, s])$ is included in $Y'$, and consequently the edge $[0, 0]\longrightarrow\phi([r, s])$ is also included in $Y'$.  Because $[[0, 0]]\longrightarrow[r, s]$ is in $Y$, the edge $[0, 0]\longrightarrow[r, s]$ must also be in $Y$, so $\phi([0, 0])\longrightarrow\phi([r, s])$ is included in $Y'$.  Thus $[0, 0]\longrightarrow\phi([r, s])$ and $\phi([0, 0])\longrightarrow\phi([r, s])$ are both included in $Y'$, and so $Y$ and $Y'$ must be complete graphs, by Proposition \ref{prop:2evenpreccomplete}.  But in complete graphs, $[0, 0]$ is the only node that does not directly succeed the starting node, so $\phi([0, 0])=[0, 0]$, contrary to assumption.
\end{proof}

We now return to the proof of Sloane's conjecture.
\begin{lemma}\label{lem:nodesofcomplete}
For all of the nodes $[r, r]$ in a complete graph $Y(g, k)$, $r\left(\frac{g-k}{k^2-1}\right)$ is an integer and $r\le k-1$.
\end{lemma}
\begin{proof}
Consider a reverse multiple that corresponds, as per Theorem \ref{thm:youngsthm}, to a path through all nodes in $Y(g, k)$.  This is a shifted-symmetric reverse multiple, and so the desired conclusion follows by Theorem 9 of Holt \cite{holtL}.  Alternatively, a proof method similar to that in Lemma \ref{lem:1intcomplete} may be used, where one proves that if $[r, r]$ satisfies the conditions then all its direct successors do as well, so that all nodes in the graph satisfy the conditions because the starting node does.
\end{proof}

We may now prove the main result of this section.
\begin{theorem}[Conjecture 3.7 of Sloane]\label{thm:completegraph}
$Y(g, k)$ is in $K_m$ if and only if there are exactly $m-1$ positive integers $s_1$, $s_2, \ldots, s_{m-1}$ such that $s_j\left(\frac{g-k}{k^2-1}\right)$ is an integer and $s_j\le k-1$ for all $1\le j\le m-1$.
\end{theorem}
\begin{proof}
We first show that the existence of the $s_j$ implies that $Y(g, k)$ is in $K_m$, and then the converse.

Suppose that there exist exactly $m-1$ positive integers $s_1$, $s_2, \ldots s_{m-1}$ less than $k$ such that $s_j\left(\frac{g-k}{k^2-1}\right)$ is an integer. Since $s_1$ exists, by Lemma \ref{lem:1intcomplete}, $Y(g, k)$ must be in $K_l$ for some $l$.  By definition graphs in $K_l$ have $l+1$ nodes, of which one is the starting node and one the node $[0,0]$ (by Lemma \ref{lem:0inall}).  The remaining $l-1$ nodes, by Lemmas \ref{lem:evencomplete} and \ref{lem:nodesofcomplete}, have the form $[r, r]$ where $r\le k-1$ is a positive integer such that $r(g-k)/(k^2-1)$.  There are only $m-1$ such positive integers, so $l \le m$.  Furthermore, for every positive integer $r\le k-1$ such that $r(g-k)/(k^2-1)$ is an integer, Theorem 9 of Holt \cite{holtL} gives a reverse multiple that corresponds via Theorem \ref{thm:youngsthm} to the path $[[0, 0]]\longrightarrow[r, r]$.  Therefore every node $[r, r]$ is in $Y(g, k)$, so that $m\le l$.\footnote{If one wishes to keep to a Young graphical perspective, it may also be shown computationally from equation (\ref{eq:eqnpair}) that every node $[r, r]$ directly succeeds the starting node.}  This completes this part of the proof.

Now suppose that $Y(g, k)$ is in $K_m$.  By Lemma \ref{lem:nodesofcomplete}, all the nodes $[r, r]$ of $Y(g, k)$ are such that $r(g-k)/(k^2-1)$ is an integer and $r\le k-1$.  In particular, we know that there exist such numbers $r$.  Suppose that exactly $l-1$ such numbers $r$ exist.  Then, by the first part of this proof, $Y(g, k)$ is in $K_l$.  Now, because $K_l\cap K_m$ is non-empty, we have $K_l=K_m$, so $l=m$.
\end{proof}

Note the similarity between the following result, as well as Proposition \ref{prop:2digit}, and Theorem 4 of Holt's second paper \cite{holtS}, which states that a $(g, k)$ reverse multiple is shifted-symmetric if and only if $\gcd(g-k, k^2-1)\ge k+1$.
\begin{corollary}\label{cor:completegraph}
$Y(g, k)$ is in $K_m$ if and only if $\floor*{\gcd(g-k, k^2-1)/(k+1)}=m-1$.
\end{corollary}
\begin{proof}
It is easily shown that there exist $m-1$ positive integers $s\le k-1$ such that $s(g-k)/(k^2-1)$ is an integer if and only if $m-1\le\gcd(g-k, k^2-1)/(k+1)<m$.
\end{proof}

The earliest paper \cite{sutcliffe} found in this topic, by Sutcliffe, is largely concerned with reverse multiples with two digits; he concludes (Theorem 2 in his paper) that, for given $g$, there exists a $k$ such that there exists a $(g, k)$ reverse multiple with two digits if and only if $g+1$ is composite.  However, he, and Kaczynski \cite{unabomber} after him, are not particularly concerned with the multiplier $k$: in their terminology, an $n$-digit $(g, k)$ reverse multiple is known only as an ``$n$-digit solution for $g$.''  Corollaries \ref{cor:evennear0} and \ref{cor:completegraph} now imply a result on two-digit reverse multiples with reference to $k$, a connection hinted at in property P7 of Sloane.
\begin{proposition}\label{prop:2digit}
A two-digit $(g, k)$ reverse multiple exists if and only if $k+1 \le \gcd(g-k, k^2-1)$.
\end{proposition}
\begin{proof}
By Theorem \ref{thm:youngsthm}, if there is a two-digit reverse multiple, then there exists a path of length $1$ in $Y(g, k)$ from the starting node to an even pivot node, so that there must be an even pivot node directly succeeding the starting node.  This even pivot node will also directly succeed the node $[0, 0]$, because, as mentioned previously, the equations yielded by equation (\ref{eq:eqnpair}) are the same for the starting node and $[0, 0]$.  This makes $Y(g, k)$ complete, by Corollary \ref{cor:evennear0}.  Corollary \ref{cor:completegraph}, now implies the desired result.

Conversely, suppose $k+1\le\gcd(g-k, k^2-1)$.  Now $Y(g, k)$ is complete by Corollary \ref{cor:completegraph}, so that there exists a path of length $1$ in the graph from the starting node to an even pivot node -- namely, from the starting node to any adjacent node -- so there exists a $(g, k)$ reverse multiple with two digits, by Theorem \ref{thm:youngsthm}.

This result could also follow by Theorem 4 of Holt \cite{holtS} and Young's theorem (Theorem \ref{thm:youngsthm} here).
\end{proof}
Indeed, some simple manipulations show that this result agrees exactly with the characterization of $g$ by Sutcliffe.

\section{Cyclic Graphs}\label{sec:cyclicgraphs}
We now include a minor result and a conjecture on another type of graph mentioned by Sloane.

\begin{definition}[Cyclic Graph]\label{def:cyclicgrpahs}
A Young graph $Y$ is called a \textit{cyclic Young graph on $n$ nodes} if the nodes of $Y$ that are neither the starting node nor $[0, 0]$ form the cyclic directed graph on $n$ nodes, and there exist two nodes $N_1$ and $N_2\neq[0, 0]$ such that $N_1$ directly precedes $N_2$,  the only direct predecessor of $[0, 0]$ (other that itself ) is $N_1$, and the only direct successor of $[0, 0]$ (other than itself) and starting node is $N_2$. The equivalence class of cyclic Young graphs on $n$ nodes is denoted $Z_n$.
\end{definition}

For examples of cyclic graphs, see Figures \ref{fig:z3} and \ref{fig:z5}.  Sloane makes no conjecture analogous to Theorems \ref{thm:1089graph} and \ref{thm:completegraph} for cyclic graphs.  We also have no such conjecture; however, examination of data for low values of $g$ suggests the following.

\begin{conjecture}\label{conj:zodd}
For $t\ge 3$ and $n\ge 2$, the Young graph $Y((n^2+n)t-n^2+n+1, n^2t-(n-1)^2)$ is in $Z_{2n-1}$.
\end{conjecture}

We now prove two minor results on similar types of graphs.

\begin{theorem}\label{thm:z3}
For $t \ge 3$ and $n \ge 2$, the Young graph $Y((n^2+n)t-1, n^2t-1)$ is in $Z_3$, with node and edge labels as shown in Figure \ref{fig:z3}.

\begin{figure}[ht]
\centering
\includegraphics[scale=0.5]{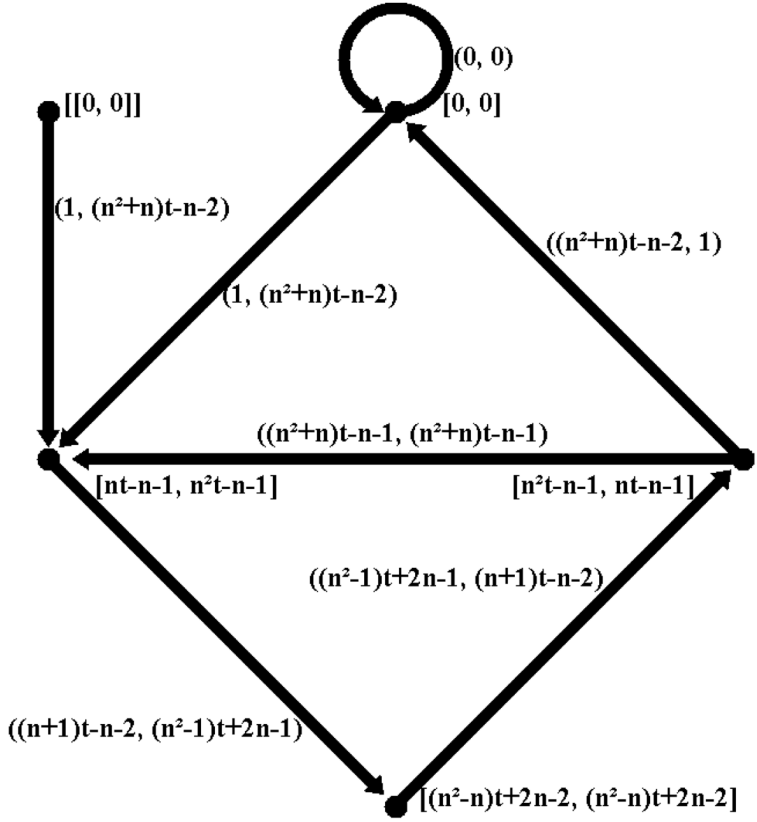}
\caption{$Y((n^2+n)t-1, n^2t-1)$ for $n\ge 2$ and $t\ge 3$}\label{fig:z3}
\end{figure}

\end{theorem}
\begin{proof}
Let $g=(n^2+n)t-1$ and $k=n^2t-1$. We compute $H(g, k)$ and thence $Y(g, k)$.

We may compute all direct successors of the starting node.  By equation (\ref{eq:eqnpair}), any node $[r_0', r_0]$ that directly succeeds the starting node by an edge $(a_0', a_0)$ must satisfy
\begin{equation}\label{eq:z31}
\begin{array}{rcl}
(n^2t-1)a_0 & = & a_0'+r_0((n^2+n)t-1) \\
(n^2t-1)a_0'+r_0' & = & a_0.
\end{array}
\end{equation}
Because $a_0\le g-1$ and $r_0' \ge 0$, the second equation of (\ref{eq:z31}) implies $a_0'\le ((n^2+n)t-2)/(n^2t-1)=1+1/n+(1/n-1)/(n^2t-1)<2$.  Recall $a_0'\neq 0$ because it is the last digit of a reverse multiple.  Therefore $a_0'=1$.  Substituting into the first equation of (\ref{eq:z31}) and taking the equation modulo $g$ gives $(n^2t-1)a_0\equiv 1\pmod{g}$, so that $a_0\equiv -(n+1)\pmod{g}$.  Since $0< a_0\le g-1$, this implies $a_0=(n^2+n)t-(n+2)$.  Substituting the values $a_0'=1$ and $a_0=(n^2+n)t-(n+2)$ into equation (\ref{eq:z31}) gives $r_0'=nt-(n+1)$ and $r_0=n^2t-(n+1)$.  Therefore the only direct successor, in $H(g, k)$, of the starting node is the node $[nt-(n+1), n^2t-(n+1)]$, and the corresponding edge is $(1, (n^2+n)t-(n+2))$.

Similar considerations for $[nt-(n+1), n^2t-(n+1)]$ show that its only direct successor in $H(g, k)$ is $[(n^2-n)t+2n-2, (n^2-n)t+2n-2]$, the corresponding edge being $((n+1)t-(n+2), (n^2-1)t+2n-1)$.  In turn, similar considerations for $[(n^2-n)t+2n-2, (n^2-n)t+2n-2]$ show that its only direct successor in $H(g, k)$ is $[n^2t-(n+1), nt-(n+1)]$, by the edge $((n^2-1)t+2n-1, (n+1)t-(n+2))$, and similar considerations for $[n^2t-(n+1), nt-(n+1)]$ show that its only two direct successors are $[nt-(n+1), n^2t-(n+1)]$, by the edge $((n^2+n)t-(n+1), (n^2+n)t-(n+1))$, and $[0, 0]$, by the edge $((n^2+n)t-(n+2), 1)$.  It may then be shown that the only direct successors of $[0, 0]$ are $[nt-(n+1), n^2t-(n+1)]$ and itself, in reasoning almost identical to that of the computation of the direct successors of the starting node.  This fully determines $H(g, k)$.  All nodes in the graph precede some pivot node, and so $Y(g, k)$ is the same labeled graph as $H(g, k)$, namely the one depicted in Figure \ref{fig:z3}.
\end{proof}

A similar result holds for cyclic graphs on five nodes.

\begin{theorem}\label{thm:z5}
For $t\ge 3$ and $n\ge 3$, the Young graph $Y((n^2+n)t-(n+2), n^2t-(n+1))$ is in $Z_5$, with node and edge labels as shown in Figure \ref{fig:z5}.
\begin{figure}[ht]
\centering
\includegraphics[scale=0.6]{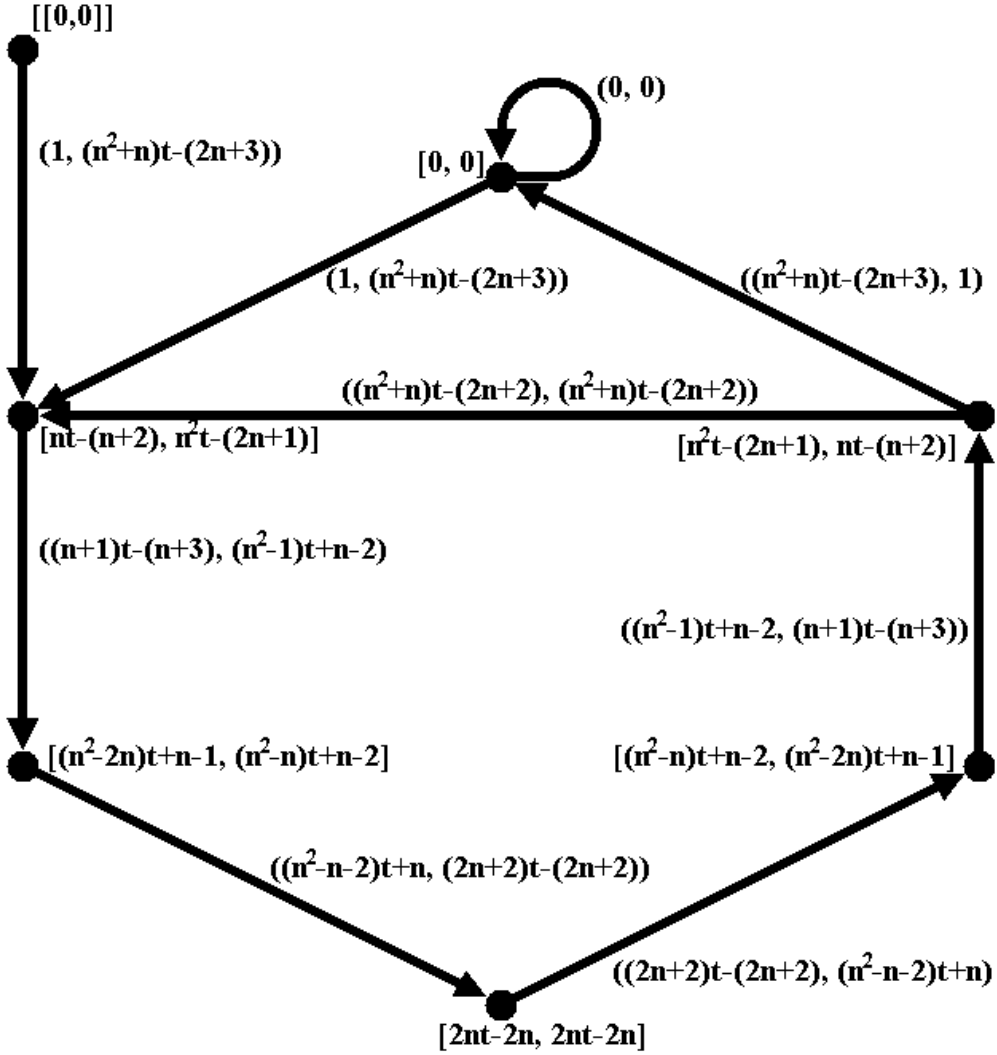}
\caption{$Y((n^2+n)t-(n+2), n^2t-(n+1))$ for $n\ge 3$ and $t\ge 3$}\label{fig:z5}
\end{figure}
\end{theorem}

\begin{proof}
The proof is similar to the proof Theorem \ref{thm:z3}. Let $g=(n^2+n)t-(n+2)$ and $k=n^2t-(n+1)$. We compute $H(g, k)$ and thence $Y(g, k)$.

We may compute all direct successors of the starting node.  By equation (\ref{eq:eqnpair}), any node $[r_0', r_0]$ that directly succeeds the starting node by an edge $(a_0', a_0)$ must satisfy
\begin{equation}\label{eq:z51}
\begin{array}{rcl}
(n^2t-(n+1))a_0 & = & a_0'+r_0((n^+n)t-(n+2)) \\
(n^2t-(n+1))a_0'+r_0' & = & a_0.
\end{array}
\end{equation}
Because $a_0\le g-1$ and $r_0'\ge 0$, the second equation of (\ref{eq:z51}) implies $a_0'\le ((n^2+n)t-(n+3))/(n^2t-(n+1))=1+1/n+(1/n-1)/(n^2t-(n+1))<2$. Recall $a_0'\neq 0$ because it is the last digit of a reverse multiple. Therefore $a_0'=1$. Substituting into the first equation of (\ref{eq:z51}) and taking the equation modulo $g$ gives $(n^2t-(n+1))a_0\equiv 1\pmod{g}$, so that $a_0\equiv -(n+1)\pmod{g}$. Since $0<a_0\le g-1$, this implies $a_0=(n^2+n)t-(2n+3)$. Substituting the values $a_0'=1$ and $a_0=(n^2+n)t-(2n+3)$ into equation (\ref{eq:z51}) gives $r_0'=nt-(n+2)$ and $r_0=n^2t-(2n+1)$. Therefore the only direct successor, in $H(g, k)$, of the starting node is node $[nt-(n+2), n^2t-(2n+1)]$, and the corresponding edge is $(1, (n^2+n)t-(2n+3))$.

Similar considerations for $[nt-(n+2), n^2t-(2n+1)]$ show that its only direct successor in $H(g, k)$ is $[(n^2-2n)t+n-1, (n^2-n)t+n-2]$, the corresponding edge being $((n+1)t-(n+3), (n^2-1)t+n-2)$. In turn, similar considerations for $[(n^2-2n)t+n-1, (n^2-n)t+n-2]$ show that its only direct successor in $H(g, k)$ is $[2nt-2n, 2nt-2n]$, by the edge $((n^2-n-2)t+n, (2n+2)t-(2n+2))$; similar considerations for $[2nt-2n, 2nt-2n]$ show that its only direct successor in $H(g, k)$ is $[(n^2-n)t+n-2, (n^2-2n)t+n-1]$, by the edge $((2n+2)t-(2n+2), (n^2-n-2)t+n)$; similar considerations for $[(n^2-n)t+n-2, (n^2-2n)t+n-1]$ show that its only direct successor in $H(g, k)$ is $[n^2t-(2n+1), nt-(n+2)]$, by the edge $((n^2-1)t+n-2, (n+1)t-(n+3))$; and similar considerations for $[n^2t-(2n+1), nt-(n+2)]$ show that its only two direct successors are $[nt-(n+2), n^2t-(2n+1)]$, by the edge $((n^2+n)t-(2n+2), (n^2+n)t-(2n+2))$, and $[0, 0]$, by the edge $((n^2+n)t-(2n+3), 1)$. It may then be shown that the only direct successors of $[0, 0]$ are $[nt-(n+2), n^2t-(2n+1)]$ and itself, in reasoning almost identical to that of the computation of the direct successors of the starting node. This fully determines $H(g, k)$. All nodes in the graph precede some pivot node, and so $Y(g, k)$ is the same labeled graph as $H(g, k)$, namely the one depicted in Figure \ref{fig:z5}.
\end{proof}

While lists of reverse multiples arising from Figures \ref{fig:z3} and \ref{fig:z5} may be given in general form, as with Figure \ref{fig:1089graph}, we omit them for the sake of convenience, as the smallest reverse multiple involved, that is, the four-digit reverse multiple corresponding to the path in $Y((n^2+n)t-1, n^2t-1)$ (Figure \ref{fig:z3}) from the starting node to $[nt-n-1, n^2t-n-1]$ to $[(n^2-n)t+2n-2, (n^2-n)t+2n-2]$, has the rather cumbersome general form of either $(1, (n+1)t-(n+2), (n^2-1)t+2n-1, (n^2+n)t-(n+2))_{(n^2+n)t-1}$ or $n^2(n+1)^4t^3-n(n^4+6n^3+12n^2+10n+3)t^2+(4n^3+10n^2+8n+2)t-(4n+4)$. In its smallest case, $n=2$ and $t=3$, the graph in Figure \ref{fig:z3} gives the (base $10$) sequence $6576$, $32314464$, $549240672$, $9336986208$, $158728660320$, $158760968208\ldots$, and the graph in Figure \ref{fig:z5}, in its smallest case, $n=t=3$, gives the sequence $34633320$, $991522582544640$, $30737199019884240$, $952853168577411840$, $29538448224860767440$, $915691894969644791040\ldots$; neither appear in the OEIS \cite{oeis}.

\section{Further Research}\label{sec:futres}

Given the scarcity of work on the problem, it is not particularly surprising that Young graphs, especially the relation between the associated number theory and graph theory, are not well understood.  This paper makes some limited progress, mostly concerning specific equivalence classes of Young graphs, from the work of Young and Sloane, but still little is known in general.  We now discuss some directions for potential further work on the problem, after which discussion we include several conjectures (Subsection \ref{subsec:conj}).

Section 5 in Sloane's paper consists of a list of open problems, which the reader may investigate, one of which asks whether there are nonisomorphic Young graphs with isomorphic underlying directed graphs.  Regarding this question, our program finds that for $g \le 336$, there are no pairs of nonisomorphic Young graphs with isomorphic underlying directed graphs; the program could not proceed past this point.  Holt also lists a number of open problems, in both algebraic and graphical lines. A particularly intriguing issue concerns a kind of graph interrelation: let us say the Young graph $Y_2$ is a \textit{Holt transform} of the graph $Y_1$ if there is a one-to-one function $\phi: (a, b)\mapsto [a, b]$ from the edges of $Y_1$ to the nodes of $Y_2$ such that if the edge $(a, b)$ leads into a node $N$ in $Y_1$, and the edge $(c, d)$ leads out of it, then the node $[a, b]$ directly succeeds the node $[c, d]$ in $Y_2$ -- the graphs $Y(14, 3)$ (Figure \ref{fig:y143}) and $Y(3, 2)$ (Figure \ref{fig:1089graph} for $k=2$ and $b=1$) are an example of such a pair. While this defines a function that most meaningfully pertains only to certain types of graphs,\footnote{For example, applied to a graph in which many edges enter and leave a single node, such as a member of $K_m$, it results in a graph that cannot be reliably traced back to the original graph, as this would require many edges between different nodes to collapse to a single node.} within that subset of graphs, considered as a map between equivalence classes, rather than the graphs themselves, it often results not only in pairs, as Holt \cite{holtAsym} most directly discussed, but sequences of transforms; for example, the class of $Y(3, 2)$ transforms to the class of $Y(14, 3)$, which transforms to the class of $Y(44, 13)$, which transforms to the class of $Y(30, 11)$, which transforms to the class of $Y(81, 32)$. Several questions naturally arise about these sequences, as to their possible lengths, their prevalence, and so on; Holt \cite{holtAsym} also demonstrates connections between such graphs and palintiple derivation.

Examination of data also suggests the following line of study: given an equivalence class $X$ of Young graphs, let $\text{Freq}_X(x)$ be the number of $Y(g, k)\in X$ with $g\le x$.  The study of these functions may shed light on the characteristics of different equivalence classes.  For example, their relative growth rates carry information about the distribution of classes amongst each other.  By Theorem \ref{thm:1089graph}, if $X$ is the set of 1089 graphs, then $\text{Freq}_X(x)$ grows roughly as $x\ln x$.  The growth rate for the set of $K_m$ graphs is presently unknown, although Corollary \ref{cor:completegraph} reduces it to the problem of computing the number of pairs $(g, k)$ with $g\le x$ and $\floor*{\gcd(g-k, k^2-1)/(k+1)}=m-1$, and they have some resemblance to functions of the form $x(\ln x)^{\alpha(m)}$.  Some graphs pertaining to this issue are included in the URL mentioned at the start of Subsection \ref{subsec:conj}.

This paper is largely devoted to characterizing a few specific equivalence classes; infinitely more remain, although it is possible that they fall into categories like the $K_m$ of complete graphs, where technically different equivalence classes follow a similar pattern. The partition of the set of pairs $(g, k)$ into equivalence classes may well be the central utility of Young's graphical approach, as it is difficult to meaningfully classify reverse multiples using only algebraic methods.  Further work may be done characterizing specific equivalence classes, for example those featured in Subsection \ref{subsec:conj}.  The graphs in the $Z_n$ present themselves as a natural next step, given the frequency of their occurrence, especially for small values of $n$, and the results of Section \ref{sec:cyclicgraphs}.  While it would be ideal to characterize all equivalence classes, this is likely a rather difficult task given present methods.  However, it might further the study of equivalence classes to study of the graphs $H(g, k)$, in addition to the graphs $Y(g, k)$ that are the focus of this paper.  It should be easier to study $H$ graphs, as their generation does not involve the additional, potentially complicating step of removing all nodes that do not precede pivot nodes.

Since Young graph isomorphism is a stronger condition than isomorphism of the underlying directed graphs, and since the set of directed graphs that may occur as the underlying directed graphs of Young graphs is a proper subset of the set of all directed graphs, Young graph isomorphism should be an easier condition to test than normal directed graph isomorphism; indeed, Theorem \ref{thm:0iso} demonstrates this concretely.  This raises the question of identifying the most efficient test for Young graph isomorphism.  Answering this question could significantly speed the computations involved in testing isomorphism, as our current program for this purpose is rather slow.  The program, given substantial time, was unable to decide whether $Y(173, 54)$ and $Y(337, 54)$ are isomorphic, whether $Y(458, 431)$ and $Y(459, 436)$ are isomorphic, and whether $Y(471, 45)$ and $Y(159, 45)$ are isomorphic; these might be good cases on which to attempt new methods.

\subsection{Conjectures}\label{subsec:conj}

A program in Java was used to generate all graphs with $g\le 336$ and to sort them into the equivalence classes determined by isomorphism.  These programs and some data on Young graphs may be found at \url{https://sites.google.com/site/younggraphs/home}; also on this site may be found a visualization of the curious absence of reverse multiples where $k\approx g$, which could be an interesting issue to investigate.  From the examination of these graphs and their isomorphism classes, we conjecture the following.

\begin{conjecture}\label{conj:y117}
For all $n\ge 0$, the Young graphs $Y(2(n+2)(n+3)-1, 2(n+2)^2-1)$, $Y(23+20n, 13+12n)$, $Y(23+15n, 17+10n)$, $Y(26+12n, 10+4n)$, and $Y(11+6n, 7+3n)$ are all isomorphic to $Y(11, 7)$.
\end{conjecture}

\begin{conjecture}\label{conj:85}
For all $n\ge 3$, $Y(n^2-1, n^2-n-1)\cong Y(8, 5)$.
\end{conjecture}

\begin{conjecture}\label{conj:143}
For all $t\ge 2$ and $n\ge 2$, $Y((n^2+n)t+n, n^2t+n-1)\cong Y(14, 3)$.
\end{conjecture}

\begin{conjecture}\label{conj:3610}
For all $n\ge 2$, $Y(17+30n, 14+25n)\cong Y(36, 10)$.
\end{conjecture}

Grimm and Ballew found an algorithm for generating three-digit reverse multiples, which goes as follows: take $d_1, d_2\in\mathbb{Z}^+$ such that $\gcd(d_1, d_2)=1$, and then take some $m\in\mathbb{Z}$ such that $d_2^2-d_1^2\equiv 0\pmod{(m-1)}$, $m\ge d_2+1$, $m>d_1+1$, and $\gcd(m+1, d_2)=1$.  Then let $k=(d_2^2-md_1^2)/(m-1)$, and let the pair $(M, N)$ be an integer solution to $d_1(m+1)M-d_2N=1$.  Then $(d_2t+kM, d_1(m+1)t+kN, md_2t+mkM+d_1)_g$ is a $(g, m)$ reverse multiple for $g=(m^2-1)t-mNd_1+d_2M(m+1)$, for any $t$ such that $g>m$.  Note that $d_1$ is $r_1$ and $d_2$ is $r_0$.  From examination of some of these multiples, we make Conjecture \ref{conj:GBiso}.

\begin{conjecture}\label{conj:GBiso}
If $g_1$ and $g_2$ are different bases obtained by performing the Grimm-Ballew procedure for the same $d_1$, $d_2$, and $m$, then $Y(g_1, m)\cong Y(g_2, m)$.
\end{conjecture}

\begin{conjecture}\label{conj:classgrowth}
Let $\text{Classes}(x)$ denote the number of equivalence classes that occur in all $Y(g, k)$ for $g\le x$.  $\text{Classes}(x)$ grows roughly proportionately to $x^2$.
\end{conjecture}
Regarding this conjecture, we find that the $r$-value for the set $\{(x, \text{Classes}(x)^{1/2})\mid x\le 336\}$ is around $0.9996$.  However, $336$ seems to be a rather low upper bound for this type of problem: for example, the functions $\text{Freq}_{K_m}(x)$, for $x\le 400$ are well-approximated by functions of the form $c\cdot x\ln x$, but for $x\le 10^5$ it becomes apparent that such an approximation fails.  Conjecture \ref{conj:classgrowth} is thus made rather hesitantly.

\section{Acknowledgements}\label{sec:acks}
Many thanks to Professor Ramin Takloo-Bighash of the University of Illinois at Chicago for suggesting this problem for research, to Mr. Paul Karafiol of the Ogden International School of Chicago and Prof. Benjamin Klaff of Tulane University for their mentoring and encouragement, to Dr. Tanya Khovanova of the Massachusetts Institute of Technology for her help with editing, and to the referee for many helpful comments.

\bibliography{YGbib}

\begin{thebibliography}{10}

\bibitem{gardner}
Martin Gardner.
\newblock {\em Mathematical Magic Show}.
\newblock Mathematical Association of America, 1990.

\bibitem{GB}
C.~A. Grimm and D.~W. Ballew.
\newblock Reversible multiples.
\newblock {\em Journal of Recreational Mathematics}, 8:89--91, 1975-1976.

\bibitem{hoey}
Dan Hoey.
\newblock Untitled online article.
\newblock \url{https://oeis.org/A008919/a008919.txt}, accessed April 10, 2015.

\bibitem{holtL}
Benjamin~V. Holt.
\newblock Some general results and open questions on palintiple numbers.
\newblock {\em Electronic Journal of Combinatorial Number Theory}, 14, August
  2014.

\bibitem{holtAsym}
Benjamin~V. Holt.
\newblock Derived palintiple families and their palinomials.
\newblock arxiv:1412.0231v3 [math.NT], March 2015.

\bibitem{holtS}
Benjamin~V. Holt.
\newblock Some thoughts on determining symmetric palintiples.
\newblock arXiv:1410.2356v2 [math.NT], March 2015.

\bibitem{oeis}
OEIS~Foundation Inc.
\newblock The on-line encyclopedia of integer sequences.
\newblock \url{http://oeis.org}, (2013).

\bibitem{unabomber}
T.~J. Kaczynski.
\newblock Note on a problem of {A}lan {S}utcliffe.
\newblock {\em Mathematics Magazine}, 41(2):84--86, March 1968.

\bibitem{klosinski}
Leonard~F. Klosinski and Dennis~C. Smolarski.
\newblock On the reversing of digits.
\newblock {\em Mathematics Magazine}, 42(4):208--210, September 1969.

\bibitem{pudwell}
Lara Pudwell.
\newblock Digit reversal without apology.
\newblock {\em Mathematics Magazine}, 80(2):129--132, April 2007.

\bibitem{sloane}
N.~J.~A. Sloane.
\newblock 2178 and all that.
\newblock {\em The Fibonacci Quarterly}, 52(2):99--120, May 2014.

\bibitem{sutcliffe}
Alan Sutcliffe.
\newblock Integers that are multiplied when their digits are reversed.
\newblock {\em Mathematics Magazine}, 39(5):282--287, November 1966.

\bibitem{WW}
Roger Webster and Gareth Williams.
\newblock On the trail of reverse divisors: 1089 and all that follow.
\newblock {\em Mathematical Spectrum}, 45(3):96, May 2013.

\bibitem{PengNumb}
David Wells.
\newblock {\em The Penguin Dictionary of Curious and Interesting Numbers}.
\newblock Penguin, illustrated, revised edition, September 1997.

\bibitem{youngRM}
Anne~Ludington Young.
\newblock $k$-reverse multiples.
\newblock {\em The Fibonacci Quarterly}, 30(2):166--174, May 1992.

\bibitem{youngTrees}
Anne~Ludington Young.
\newblock Trees for $k$-reverse multiples.
\newblock {\em The Fibonacci Quarterly}, 30(2):126--132, May 1992.

\end{thebibliography}

\end{document}